\documentclass{article}

\usepackage[utf8]{inputenc} % allow utf-8 input
\usepackage[T1]{fontenc}    % use 8-bit T1 fonts
\usepackage{hyperref}       % hyperlinks
\usepackage{url}            % simple URL typesetting
\usepackage{booktabs}       % professional-quality tables
\usepackage{nicefrac}       % compact symbols for 1/2, etc.
\usepackage{microtype}      % microtypography
\usepackage{graphicx}

\usepackage{amsmath}
\usepackage{fourier}
\usepackage{braket}
\usepackage[standard,amsmath,thmmarks]{ntheorem}

\newtheorem{assumption}{Assumption}

\usepackage{algorithm}
\usepackage{algpseudocode} 

\usepackage{bbold}

\usepackage{authblk}

\newcommand{\norm}[1]{\Vert #1 \Vert}

\newcommand{\du}{\mathrm{d}}
\newcommand{\eu}{\mathrm{e}}

\DeclareMathOperator{\dom}{\mathrm{dom}}

\DeclareMathOperator{\inte}{\mathrm{int}}
\DeclareMathOperator*{\argmin}{ \arg \min }

\DeclareMathOperator{\tr}{\mathrm{Tr}}

\newcommand{\fqst}{f_{\text{QST}}}

\begin{document}

\title{A General Convergence Result for \\ Mirror Descent with Armijo Line Search}

\author[1]{Yen-Huan~Li}
\author[1,2]{Carlos A. Riofr\'{\i}o}
\author[1]{Volkan~Cevher}

\affil[1]{\'{E}cole polytechnique f\'{e}d\'{e}rale de Lausanne, Switzerland}
\affil[2]{Freie Universit\"{a}t Berlin, Germany}

\date{}

% \author{Yen-Huan~Li, Carlos A. Riofrio, and Volkan Cevher}

\maketitle

\begin{abstract}
Existing convergence guarantees for the mirror descent algorithm require the objective function to have a bounded gradient or be smooth relative to a Legendre function. 
The bounded gradient and relative smoothness conditions, however, may not hold in important applications, such as quantum state tomography and portfolio selection. 
In this paper, we propose a local version of the relative smoothness condition as a generalization of its existing global version, and prove that under this local relative smoothness condition, the mirror descent algorithm with Armijo line search always converges. 
Numerical results showed that, therefore, the mirror descent algorithm with Armijo line search was the fastest guaranteed-to-converge algorithm for quantum state tomography, empirically on real data-sets. 
\end{abstract}

\section{Introduction}

Consider a constrained convex optimization problem: 
\begin{equation}
f^\star = \min_{x} \set{ f ( x ) | x \in \mathcal{X} } , \tag{P} \label{eq_problem_general}
\end{equation}
where $f$ is a convex differentiable function, and $\mathcal{X}$ is a convex closed set in $\mathbb{R}^d$. 
We assume that $f^\star > - \infty$. 
% We study the convergence of the mirror descent algorithm, assuming that the objective function $f$ is \emph{locally relatively smooth}. 

The mirror descent algorithm is standard for solving such a constrained convex optimization problem \cite{Beck2003,Nemirovsky1983}. 
Given an initial iterate $x_0 \in \mathcal{X}$, the mirror descent algorithm iterates as
\begin{equation}
x_{ k + 1 } = \argmin_x \Set{ \braket{ \nabla f ( x_k ), x - x_k } + \alpha_k D_h ( x, x_k ) | x \in \mathcal{X} } \, , \quad \forall k \in \mathbb{N} \, , \label{eq_mirror_descent}
\end{equation}
for some convex differentiable function $h$ and a properly chosen sequence of step sizes $\set{ \alpha_k }$, where $D_h$ denotes the Bregman divergence induced by $h$: 
\begin{equation}
D_h ( z_2, z_1 ) := h ( z_2 ) - \left[ h ( z_1 ) + \braket{ \nabla h ( z_1 ), z_2 - z_2 } \right] , \quad \forall ( z_2, z_2 ) \in \dom h \times \dom \nabla h \, . \notag
\end{equation}
With a proper choice of the funciton $h$, the mirror descent algorithm can have an almost dimension-independent convergence rate guarantee, or lower per-iteration computational complexity. 
A famous example is the exponentiated gradient method, which enjoys both benefits \cite{Juditsky2012,Kivinen1997}. 
The exponentiated gradient method corresponds to the mirror descent algorithm with $h$ being the negative Shannon entropy. 

Convergence of the mirror descent algorithm has been established under the following two conditions on the objective function. 
\begin{enumerate}
\item Bounded gradient: There exists some $L > 0$, such that
\begin{equation}
\norm{ \nabla f ( x ) } \leq L, \quad \forall x \in \mathcal{X} , \notag
\end{equation}
for some norm $\norm{ \cdot }$ \cite{Beck2003,Nemirovsky1983}\footnote{To be precise, results in this direction assume that there exists a function $g$ satisfying
\begin{equation}
g ( x ) \in \partial f ( x ) \, , \quad \norm{ g ( x ) } \leq L \, , \quad \forall x \in \mathcal{X} \, . \notag
\end{equation}
where $\partial f ( x )$ denotes the sub-differential of $f$ at the point $x$.} 
\item Relative smoothness: There exist some $L > 0$ and a convex differentiable function $h$, such that
\begin{equation}
f ( y ) \leq f ( x ) + \braket{ \nabla f ( x ), y - x } + L D_h ( y, x ) , \quad \forall x, y \in \mathcal{X} , \notag
\end{equation}
where $D_h$ denotes the Bregman divergence induced by $h$ \cite{Auslender2006,Bauschke2017,Lu2018}. 
\end{enumerate}
These conditions may not hold, or introduce undesirable computational burdens for some applications. 
Quantum state tomography is one such instance. 

\begin{example} \label{exp_qst}
\emph{Quantum state tomography (QST)} is the task of estimating the state of qubits (quantum bits) given measurement outcomes \cite{Paris2004}; 
this task is essential to calibrating quantum computation devices. 
Numerically, it corresponds to minimizing the function
\begin{equation}
f_{\text{QST}} ( x ) := - \sum_{i = 1}^n \log \tr ( M_i x ) , \notag
\end{equation}
for given positive semi-definite matrices $M_i$, on the set of \emph{quantum density matrices}
\begin{equation}
\mathcal{D} := \Set{ x \in \mathbb{C}^{d \times d} | x \geq 0, \tr ( x ) = 1 } . \label{eq_density_matrix}
\end{equation}
The dimension $d$ equals $2^q$, where $q$ is the number of qubits (quantum bits). 
\end{example}

Notice that the diagonal of a density matrix in $\mathbb{R}^{d \times d}$ must belong to the probability simplex in $\mathbb{R}^d$; 
therefore, a density matrix can be viewed as a matrix analogue of a probability distribution. 
Regarding this observation, it is natural to consider the matrix version of the exponentiated gradient method, for which the Shannon entropy is replaced by its matrix analogue called the von Neumann entropy \cite{Bubeck2015a,Tsuda2005}. 
Unfortunately, the following is easily checked. 

\begin{proposition} \label{prop_not_applicable}
The gradient of the function $f_{\text{QST}}$ is not bounded. 
The function $f_{\text{QST}}$ is not smooth relative to the von Neumann entropy. 
\end{proposition}

A proof is given in Section \ref{proof_not_applicable}. 

Another popular choice of the function $h$ is Burg's entropy. 
The resulting mirror descent algorithm iterates as
\begin{equation}
x_{k + 1} = \left( x_k^{-1} + \alpha_k \nabla f ( x_k ) \right)^{-1} \, , \quad \forall k \in \mathbb{N} \, , \notag
\end{equation}
where $\alpha_k$ is chosen such that $\tr ( x_{k + 1} ) = 1$ \cite{Kulis2009}. 
The numerical search for $\alpha_k$ yields high per-iteration computational complexity of the mirror descent algorithm. 

We note that in terms of the objective functions and constraint sets, positron emission tomography, optimal portfolio selection, and non-negative linear inverse problems are essentially vector analogues of QST \cite{Byrne2001,Cover1984,Vardi1985}. 
The same issues we have discussed above remain in these applications, though the computational burden due to the Burg entropy may be relatively minor in these vector analogues. 

To address ``non-standard'' applications like QST, we relax the condition on the objective function. 
Specifically, we propose a novel localized version of the relative smoothness condition. 
The local relative smoothness condition does not involve any parameter, in comparison to the bounded gradient and (global) relative smoothness conditions. 
Therefore, we do not seek for a closed-form expression for the step sizes; instead, we consider selecting the step sizes adaptively by Armijo line search. 

\subsection{Related work}

The mirror descent algorithm was introduced in \cite{Nemirovsky1983}. 
The formulation \eqref{eq_mirror_descent} was proposed in \cite{Beck2003}, which is equivalent to the original one under standard assumptions. 
The interior gradient method studied in \cite{Auslender2006} is also of the form \eqref{eq_mirror_descent}; the difference lies in the technical conditions. 
Standard convergence analyses of the mirror descent, as discussed above, assume either bounded gradient or relative smoothness \cite{Auslender2006,Bauschke2017,Beck2003,Lu2018,Nemirovsky1983}. 
The exponentiated gradient method was proposed in \cite{Kivinen1997}; it is also known as the entropic mirror descent \cite{Beck2003}. 

For quantum state tomography, there are few guaranteed-to-converge optimization algorithms. 
The $R \rho R$ algorithm was proposed as an analogue of the expectation maximization (EM) algorithm \cite{Hradil1997}, but does not always converge \cite{Rehacek2007}. 
The diluted $R \rho R$ algorithm is a variant of the $R \rho R$ algorithm; it guarantees convergence by exact line search \cite{Rehacek2007}. 
The Frank-Wolfe algorithm converges with a step size selection rule slightly different from the standard one \cite{Odor2016}. 
The SCOPT algorithm proposed in \cite{Tran-Dinh2015b}, a proximal gradient method for composite self-concordant minimization, also converges, as the logarithmic function is a standard instance of a self-concordant function. 
The numerical results in Section \ref{sec_numerical}, unfortunately, showed that the convergence speeds of the diluted $R \rho R$, Frank-Wolfe, and SCOPT algorithms are not satisfactory on real data-sets. 

For the vector analogues of QST mentioned above, the standard approach is the EM algorithm \cite{Cover1984,Csiszar1984,Vardi1985}. 
The algorithm is also known as the Richardson-Lucy (RL) algorithm in astronomy and microscopy (see, e.g., \cite{Bertero2009}). 
The numerical results in Section \ref{sec_numerical} showed that the EM algorithm is slow on real data-sets for portfolio selection. 
There are faster accelerated versions of the EM algorithm based on line search, but they lack convergence guarantees \cite{Bertero2009}.  
Guranteed-to-converge variable metric methods with line search were proposed in \cite{Bonettini2016,Bonettini2009}, but they involve an infinite number of parameters to be properly tuned. 

Armijo line search was proposed in \cite{Armijo1966}, for minimizing functions with Lipschitz gradients. 
The formulation of Armijo line search studied in this paper is the generalized version proposed in \cite{Bertsekas1976}. 

\subsection{Contributions}

We propose a novel \emph{local relative smoothness condition}, and show that the condition is satisfied by a large class of objective functions. 
The main result is Theorem \ref{thm_main_MD}, which establishes convergence of the mirror descent algorithm with Armijo line search under the local relative smoothness condition. 
Numerical results showed that, because of Theorem \ref{thm_main_MD}, the exponentiated gradient method with Armijo line search was the fastest guaranteed-to-converge algorithm for QST, empirically on real data-sets. 
To the best of our knowledge, even for globally relatively smooth objective functions, convergence of mirror descent with Armijo line search has not been proven; 
Theorem \ref{thm_main_MD} provides the first convergence guarantee for this setup. 

%We review the setting of mirror descent with Armijo line search in Section \ref{sec_MD}. 
%We define the relative smoothness condition formally, and provide some illustrative examples in Section \ref{sec_relative_smoothness}. 
%We present the main result, the convergence guarantee, together with its proof in Section \ref{sec_main}. 
%We show the numerical results for QST on real data-sets in Section \ref{sec_numerical}

\section{Mirror Descent with Armijo Line Search} \label{sec_MD}

Let $h$ be a convex differentiable function strictly convex on $\mathcal{X}$. 
The corresponding Bregman divergence is given by
\begin{equation}
D_h ( z_2, z_1 ) := h ( z_2 ) - \left[ h ( z_1 ) + \braket{ \nabla h ( z_1 ), z_2 - z_2 } \right] , \quad \forall ( z_2, z_2 ) \in \dom h \times \dom \nabla h . \notag
\end{equation}
Because of the strict convexity of $h$, it holds that $D_h ( z_2, z_1 ) \geq 0$, and $D_h ( z_2, z_1 ) = 0$ if and only if $z_2 = z_1$. 

Define $\tilde{\mathcal{X}} := \mathcal{X} \cap \dom \nabla f \cap \dom \nabla h$. 
The corresponding mirror descent algorithm starts with some $x_0 \in \tilde{\mathcal{X}}$, and iterates as
\begin{equation}
x_k = x_{k - 1} ( \alpha_k ) := \argmin_x \Set{ \alpha_k \braket{ \nabla f ( x_{k - 1} ), x - x_{k - 1} } + D_h ( x, x_{k - 1} ) | x \in \mathcal{X} } , \quad \forall k \in \mathbb{N} , \notag
\end{equation}
where $\alpha_k$ denotes the step size. 
To ensure that the mirror descent algorithm is well-defined, we will assume the following throughout this paper. 

\begin{assumption} \label{assump_MD}
For every $x \in \tilde{\mathcal{X}}$ and $\alpha \geq 0$, $x ( \alpha )$ is uniquely defined and lies in $\tilde{\mathcal{X}}$. 
\end{assumption}

There are several sufficient conditions that guarantee Assumption \ref{assump_MD}, but in practice, it is typically easier to directly check Assumption \ref{assump_MD}. 
The interested reader is referred to, e.g., \cite{Bauschke2017,Bauschke2001} for the details. 

\begin{algorithm}[t]

\caption{Mirror Descent with Armijo Line Search}

\label{alg_MD}

\begin{algorithmic}[1]
\Require 
$\bar{\alpha} > 0$, $r \in ( 0, 1 )$, $\tau \in ( 0, 1 )$, $x_0 \in \mathcal{X}_h$
% \State Set $\rho_0 \leftarrow \mathrm{Id} / d$, where $\mathrm{Id}$ denotes the identity matrix. 
\For{$k = 1, 2, \ldots$}
% \State $G_k \leftarrow f' ( \rho_{k} )$
\State $\alpha_k \leftarrow \bar{\alpha}$
\While{$ \tau \Braket{ \nabla f ( x_{k - 1} ) , x_{k - 1} ( \alpha_k ) - x_{k - 1} } + f ( x_{k - 1} ) < f( x_{k - 1} ( \alpha_k ) ) $}
\State $\alpha_k \leftarrow r \alpha_k$
\EndWhile
\State $x_{k} \leftarrow x_{k - 1} ( \alpha_k )$
\EndFor
\end{algorithmic}

\end{algorithm}

We consider choosing the step sizes by the Armijo rule. 
Let $\bar{\alpha} > 0$ and $r, \tau \in ( 0, 1 )$. 
The Armijo rule outputs $\alpha_k = r^j \bar{\alpha}$ for every $k$, where $j$ is the least non-negative integer such that 
\begin{equation}
f ( x_{k - 1} ( r^j \bar{\alpha} ) ) \leq f ( x_{k - 1} ) + \tau \braket{ \nabla f ( x_{k - 1} ), x_{k - 1} ( r^j \bar{\alpha} ) - x_{k - 1} } . \notag
\end{equation}
The Armijo rule can be easily implemented by a while-loop, as shown in Algorithm \ref{alg_MD}. 

\section{Local Relative Smoothness} \label{sec_relative_smoothness}

In this section, we introduce the local relative smoothness condition, and provide a detailed discussion. 
In particular, we provide some practical approaches to checking the local relative smoothness condition, alone with concrete examples illustrating when the practical approaches can and cannot be applied. 

Roughly speaking, the local relative smoothness condition asks that for every point, there exists a neighborhood on which $f$ is relatively smooth. 

\begin{definition} \label{def_local_smoothness}
We say that $f$ is locally smooth relative to $h$ on $\mathcal{X}$, if for every $x \in \mathcal{X} \cap \dom f$, there exist some $L_x > 0$ and $\varepsilon_x > 0$, such that 
\begin{equation}
f ( z_2 ) \leq f ( z_1 ) + \braket{ \nabla f ( z_1 ), z_2 - z_1 } + L_x D_h ( z_2, z_1 ) , \quad \forall z_1, z_2 \in \mathcal{B}_{\varepsilon_x} ( x ) \cap \tilde{\mathcal{X}} , \label{eq_local_smoothness}
\end{equation}
where $\mathcal{B}_{\varepsilon_x} ( x )$ denotes the ball centered at $x$ of radius $\varepsilon_x$ with respect to a norm. 
\end{definition}

%\begin{definition}
%We say that $f$ is locally smooth relative to $h$, if for every $x \in \tilde{\mathcal{X}}$, there exists some $L_x > 0$ and $\varepsilon_x > 0$, such that 
%\begin{equation}
%f ( z_2 ) \leq f ( z_1 ) + \braket{ \nabla f ( z_1 ), z_2 - z_1 } + L_x D_h ( z_2, z_1 ) , \quad \forall z_1, z_2 \in \mathcal{B}_{\varepsilon_x} ( x ) \cap \tilde{\mathcal{X}} , \notag
%\end{equation}
%where $\mathcal{B}_{\varepsilon_x} ( x )$ denotes the ball centered at $x$ of radius $\varepsilon_x$ with respect to a norm. 
%\end{definition}

If we set $h: x \mapsto ( 1 / 2 ) \norm{ x }_2^2$, then \eqref{eq_local_smoothness} becomes
\begin{equation}
f ( z_2 ) \leq f ( z_1 ) + \braket{ \nabla f ( z_1 ), z_2 - z_1 } + \frac{L_x}{2} \norm{ z_2 - z_1 }_2^2 , \quad \forall z_1, z_2 \in \mathcal{B}_{\varepsilon_x} ( x ) \cap \tilde{\mathcal{X}} , \notag
\end{equation}
This is indeed the the locally Lipschitz gradient condition in literature. 

\begin{lemma} \label{lem_equiv}
The following two statements are equivalent. 
\begin{enumerate}
\item The function $f$ is locally smooth relative to $h: x \mapsto ( 1 / 2 ) \norm{ x }_2^2$ on $\mathcal{X}$. 
\item Its gradient $\nabla f$ is locally Lipschitz on $\inte \tilde{\mathcal{X}}$; that is, for every $x \in \mathcal{X} \cap \dom f$, there exists some $L_x > 0$ and $\varepsilon_x > 0$, such that 
\begin{equation}
\norm{ \nabla f ( z_2 ) - \nabla f ( z_1 ) }_2 \leq L_x \norm{ z_2 - z_1 }_2, \quad \forall z_1, z_2 \in \mathcal{B}_{\varepsilon_x} ( x ) \cap \tilde{\mathcal{X}} . \notag
\end{equation}
\end{enumerate}
\end{lemma}

The proof of Lemma \ref{lem_equiv} is standard; we give it in Appendix \ref{app_lem_equiv}. 

It is already known that the local Lipschitz gradient condition lies strictly between the following two conditions. 
\begin{enumerate}
\item The function $f$ is differentiable. 
\item The gradient of $f$ is (globally) Lipschitz. 
\end{enumerate}
See \cite{Hiriart-Urruty1984,Ioffe2002} for the details. 

The following result provides a practical approach to checking the local Lipschitz gradient condition. 

\begin{proposition} \label{prop_twice_differentiability}
Suppose that $\dom f \cap \mathcal{X} $ is relatively open in $\mathcal{X}$, and $f$ is twice continuously differentiable on $\dom f \cap \mathcal{X}$. 
Then $f$ is locally smooth relative to $h ( \cdot ) := ( 1 / 2 ) \norm{ \cdot }_2^2$ on $\mathcal{X}$. 
\end{proposition}

\begin{proof}
Recall the definition of relative openness: 
For every $x$ in $\dom f \cap \mathcal{X}$, there exists some $\varepsilon_x$ such that $\mathcal{B}_{\varepsilon_x} ( x ) \cap \mathcal{X} \subseteq \dom f \cap \mathcal{X}$. 
Notice that the largest eigenvalue of $\nabla^2 f$ is a continuous function on $\mathcal{B}_{\varepsilon_x} ( x ) \cap \mathcal{X}$; by the extreme value theorem, there exists some $L_x$ such that $\nabla^2 f ( z ) \leq L_x I$ for every $z \in \mathcal{B}_{\varepsilon_x} ( x ) \cap \mathcal{X}$. 
For every $z_1, z_2 \in \mathcal{B}_{\varepsilon_x} \cap \tilde{\mathcal{X}}$, we use Taylor's formula with the integral remainder and write
\begin{align}
f ( z_2 ) & = f ( z_1 ) + \braket{ \nabla f ( z_1 ), z_2 - z_1 } + \int_0^1 \int_0^t \braket{ \nabla^2 f ( z_1 + \tau ( z_2 - z_1 ) ) ( z_2 - z_1 ), z_2 - z_1 } \, \du \tau \, \du t \notag \\
& \leq f ( z_1 ) + \braket{ \nabla f ( z_1 ), z_2 - z_1 } + \int_0^1 \int_0^t L_x \norm{ z_2 - z_1 }_2^2 \, \du \tau \, \du t \notag \\
& = f ( z_1 ) + \braket{ \nabla f ( z_1 ), z_2 - z_1 } + \frac{L_x}{2} \norm{ z_2 - z_1 }_2^2 ,  \notag
\end{align}
which proves the proposition. 
\end{proof}

\begin{corollary} \label{cor_practical}
If $f$ is twice continuously differentiable on $\mathcal{X}$, then it is locally smooth relative to $h ( \cdot ) := ( 1 / 2 ) \norm{ \cdot }_2^2$ on $\mathcal{X}$. 
\end{corollary}

Indeed, under the setting of Corollary \ref{cor_practical}, the function $f$ has a bounded Hessian by the extreme value theorem, and hence is smooth relative to $h ( \cdot ) := ( 1 / 2 ) \norm{ \cdot }_2^2$, i.e., the function satisfies the standard smoothness assumption in literature \cite{Nesterov2004}; 
then most existing convergence results for first-order optimization algorithms apply. 
To derive an upper bound of the Lipschitz parameter, however, may be non-trivial. 
Moreover, there are cases where Corollary \ref{cor_practical} does not apply, while Proposition \ref{prop_twice_differentiability} is applicable. 
Below is an example. 

\begin{example} \label{exp_exp_linear}
Set $f ( x ) := - \log ( x_1 ) - \log ( x_2 )$ for every $x := ( x_1, x_2 ) \in \mathbb{R}^2$. 
Set $\mathcal{X}$ to be the positive orthant. 
Then $f$ is not twice continuously differentiable on $\mathcal{X}$; for example, $\nabla^2 f ( 1, 0 )$ does not exist. 
However, Proposition \ref{prop_twice_differentiability} is applicable---$\dom f \cap \mathcal{X}$ is relatively open in $\mathcal{X}$ as $\dom f$ is open, and it is easily checked that $f$ is twice continuously differentiable on $\dom f \cap \mathcal{X}$. 
\end{example}

Note that the local Lipschitz gradient condition is not always applicable. 

\begin{example}
Set $f( x ) := x_1 \log(x_1) + x_2 \log(x_2)$ for every $x := ( x_1, x_2 ) \in \mathbb{R}^2$, where we adopt the convention that $0 \log 0 := 0$. 
Set $\mathcal{X}$ to be the probability simplex in $\mathbb{R}^2$. 
Then $f$ is not locally smooth relative to $h ( \cdot ) := ( 1 / 2 ) \norm{ \cdot }_2^2$. 
For example, the point $x = ( 0, 1 )$ lies in $\dom f \cap \mathcal{X}$, while $\nabla f$ is unbounded around $( 0, 1 )$. 
However, it is obvious that $f$ is locally smooth relative to the negative Shannon entropy---indeed, $f$ itself is the negative Shannon entropy function. 
\end{example}

A standard setting for the mirror descent algorithm requires the following \cite{Auslender2006,Beck2003,Juditsky2012}. 

\begin{assumption} \label{assump_standard_MD}
The function $h$ is strongly convex with respect to a norm $\norm{ \cdot }$ on $\mathcal{X}$; that is, there exists some $\mu > 0$, such that 
\begin{equation}
D_h ( z_2, z_1 ) \geq \frac{\mu}{2} \norm{ z_2 - z_1 }^2, \quad \forall ( z_2, z_1 ) \in ( \dom h \cap \mathcal{X} ) \times ( \dom \nabla h \cap \mathcal{X} ) . \notag
\end{equation}
\end{assumption}

If $f$ is locally smooth relative to $h ( x ) := ( 1 / 2 ) \norm{ x }_2^2$, it is also locally smooth relative to any function $\tilde{h}$ strongly convex on $\mathcal{X}$ with respect to a norm $\norm{ \cdot }$---if for some $L > 0$ and $z_1, z_2 \in \dom \nabla \tilde{h} \times \dom \tilde{h}$, it holds that
\begin{equation}
f ( z_2 ) \leq f ( z_1 ) + \braket{ \nabla f ( z_1 ), z_2 - z_1 } + \frac{L}{2} \norm{ z_2 - z_1 }_2^2 , \notag
\end{equation}
then we have
\begin{equation}
f ( z_2 ) \leq f ( z_1 ) + \braket{ \nabla f ( z_1 ), z_2 - z_1 } + \frac{C L}{\mu} D_{\tilde{h}} ( z_2, z_1 ) , \notag
\end{equation}
for some $C > 0$ such that $\norm{ \cdot }_2 \leq C \norm{ \cdot }$, which exists because all norms on a finite-dimensional space are equivalent. 
Therefore, with Assumption \ref{assump_standard_MD}, it suffices to check for local smoothness relative to $h ( x ) := ( 1 / 2 ) \norm{ x }_2^2$. 

\begin{example}
Suppose that the constraint set $\mathcal{X}$ is the probability simplex. 
By Pinsker's inequality, the negative Shannon entropy is strongly convex on $\mathcal{X}$ with respect to the $\ell_1$-norm \cite{Csiszar2011}. 
By the discussion above and Corollary \ref{cor_practical}, any convex objective function that is twice continuously differentiable on $\mathcal{X}$ is locally smooth relative to the negative Shannon entropy. 
\end{example}

It is possible that Assumption \ref{assump_standard_MD} does not hold, while we have local relative smoothness. 

\begin{example}
Consider the function $f$ as defined in Example \ref{exp_exp_linear}. 
Set $h := f$, the Burg entropy. 
Then obviously, $f$ is smooth---and hence locally smooth---relative to $h$. 
However, if we set $\mathcal{X}$ to be the positive orthant, $h$ is not strongly convex on $\mathcal{X}$. 
\end{example}

\section{Main Result} \label{sec_main}

The main result of this paper, the following theorem, says that the mirror descent algorithm with Armijo line search is well-defined, and guaranteed to converge, given assumptions discussed above. 

\begin{theorem} \label{thm_main_MD}
Suppose that Assumption \ref{assump_MD} holds. 
Suppose that $\dom f \cap \mathcal{X} \subseteq \dom h \cap \mathcal{X}$, and $f$ is locally smooth relative to $h$. 
Then the following hold. 
\begin{enumerate}
\item The Armijo line search procedure terminates in finite steps. 
\item The sequence $\set{ f ( x_k ) }$ is non-increasing.
\item The sequence $\set{ f ( x_k ) }$ converges to $f^\star$, if $\set{ x_k }$ is bounded. 
\end{enumerate}
\end{theorem}

Boundedness of the sequence $\set{ x_k }$ holds, for example, when the constraint set $\mathcal{X}$ or level set $\set{ x \in \mathcal{X} | f ( x ) \leq f ( x_0 ) }$ is bounded. 
A sufficient condition for the latter case is \emph{coercivity}---a function is called coercive, if for every sequence $\set{ x_k }$ such that $\norm{ x_k } \to + \infty$, we have $f ( x_k ) \to + \infty$ (see, e.g., \cite{Bauschke2011}). 

\section{Proof of Theorem \ref{thm_main_MD}}

The proof of Theorem \ref{thm_main_MD} stems from standard arguments (see, e.g., \cite{Auslender2006}), showing that the mirror descent algorithm converges, as long as the step sizes $\alpha_k$ are bounded away from zero. 
However, without any global parameter of the objective function, we are not able to provide an explicit lower bound for all step sizes as in \cite{Auslender2006}. 
We solve this difficulty by proving the \emph{existence} of a strictly positive lower bound, for \emph{all but a finite number} of the step sizes. 

The following result shows that for every $x \in \tilde{\mathcal{X}}$, $x( \alpha )$ can be arbitrarily close to $x$ by setting $\alpha$ very small. 
This result is so fundamental in our analysis that we will use it without explicitly mentioning it. 

\begin{lemma}
The function $x( \alpha )$ is continuous in $\alpha$ for every $x \in \tilde{\mathcal{X}}$. 
\end{lemma}

\begin{proof}
Apply Theorem 7.41 in \cite{Rockafellar2009}. 
\end{proof}

For ease of presentation, we put the proofs of some technical lemmas in Section \ref{app_auxiliary_MD}. 

\subsection{Proof of Statement 1} 

Statement 1 follows from the following lemma. 

\begin{lemma} \label{lem_finite_termination}
For every $x \in \tilde{\mathcal{X}}$, there exists some $\alpha_x > 0$, such that 
\begin{equation}
f ( x ( \alpha ) ) \leq f ( x ) + \tau \braket{ \nabla f ( x ), x ( \alpha ) - x } , \quad \forall \alpha \in ( 0, \alpha_x ] . \label{eq_armijo_to_be_proved}
\end{equation}
\end{lemma}

\begin{proof}
We write \eqref{eq_armijo_to_be_proved} equivalently as
\begin{equation}
f ( x( \alpha ) ) - \left[ f ( x ) + \braket{ \nabla f ( x ), x( \alpha ) - x } \right] \leq - ( 1 - \tau ) \braket{ \nabla f ( x ), x( \alpha ) - x } , \quad \forall \alpha \in ( 0, \alpha_x ] . \notag
\end{equation}
By the local relative smoothness condition, it suffices to check 
\begin{equation}
L_x D_h ( x( \alpha ), x ) \leq - ( 1 - \tau ) \braket{ \nabla f ( x ), x ( \alpha ) - x }, \quad \forall \alpha \in ( 0, \alpha_x ] . \notag
\end{equation}
By Lemma \ref{lem_monotonicity}, it suffices to check 
\begin{equation}
\alpha L_x D_h ( x( \alpha ), x ) \leq ( 1 - \tau ) D_h ( x( \alpha ), x ), \quad \forall \alpha \in ( 0, \alpha_x ] . \notag
\end{equation}
If $D_h ( x ( \alpha ), x ) > 0$, it suffices to set $\alpha_x = L_x^{-1} ( 1 - \tau )$. 
Otherwise, we have $x = x( \alpha )$; then Lemma \ref{lem_fixed_point_MD} implies that $x$ is a minimizer, and Lemma \ref{lem_finite_termination} follows with any $\alpha_x > 0$. 
\end{proof}

\subsection{Proof of Statements 2 and 3} 

We start with the following known result. 

\begin{theorem} \label{thm_base}
Let $\set{ x_k }$ be a sequence in $\tilde{\mathcal{X}}$. 
Suppose that the assumptions in Theorem \ref{thm_main_MD} hold. 
Then the sequence $\set{ f ( x_k ) }$ monotonically converges to $f^\star$, if the following hold. 
\begin{enumerate}
\item There exists some $\tau \in ( 0, 1 )$, such that 
\begin{equation}
f ( x_k ) \leq f ( x_{k - 1} ) + \tau \braket{ \nabla f ( x_{k - 1} ), x_k - x_{k - 1} }, \quad \forall k \in \mathbb{N} . \notag
\end{equation}
\item The sum of step sizes diverges, i.e., $\sum_{k = 1}^\infty \alpha_k = + \infty$. 
\end{enumerate}
\end{theorem}

Theorem \ref{thm_base} is essentially a restatement of Theorem 4.1 in \cite{Auslender2006}. 
We give a proof in Appendix \ref{app_thm_base} for completeness. 

The first condition in Theorem \ref{thm_main_MD} is automatically satisfied by the definition of Armijo line search. 
The second condition is verified by the following lemma. 

\begin{lemma} \label{lem_lower_bound}
Suppose that the assumptions in Theorem \ref{thm_main_MD} hold. 
If none of the iterates is a solution to \eqref{eq_problem_general}, it holds that $\sum_{k = 1}^{\infty} \alpha_k = + \infty$. 
\end{lemma}

\begin{proof}
We prove by contradiction. 
Suppose that $\liminf \set{ \alpha_k } = 0$. 
Then there exists a sub-sequence $\set{ \alpha_k | k \in \mathcal{K} \subseteq \mathbb{N} }$ converging to zero. 
By the boundedness of $\set{ x_k }$, there exists a sub-sequence $\set{ x_k | k \in \mathcal{K}' - 1 }$ converging to a limit point $x_\infty$, for some $\mathcal{K}' \subseteq \mathcal{K}$. 
Notice that $\set{ \alpha_k | k \in \mathcal{K}' }$ converges to zero. 
For large enough $k \in \mathcal{K}' - 1$, we have
\begin{equation}
f ( x_{k - 1} ( r^{-1} \alpha_k ) ) > f ( x_{k - 1} ) + \tau \braket{ \nabla f ( x_{k - 1} ), x_{k - 1} ( r^{-1} \alpha_k ) - x_{k - 1} } , \notag
\end{equation}
which implies 
\begin{align}
& f ( x_{k - 1} ( r^{-1} \alpha_k ) ) - \left[ f ( x_{k - 1} ) + \braket{ \nabla f ( x_{k - 1} ), x_{k - 1} ( r^{-1} \alpha_k ) - x_{k - 1} } \right] \notag \\
& \quad > - ( 1 - \tau ) \braket{ \nabla f ( x_{k - 1} ), x_{k - 1} ( r^{-1} \alpha_k ) - x_{k - 1} } . \notag
\end{align}
By the local relative smoothness condition and Lemma \ref{lem_monotonicity}, we write 
\begin{equation}
r^{-1} \alpha_k L_{x_\infty} D_h ( x_{k - 1} ( r^{-1} \alpha_k ), x_{k - 1} ) > ( 1 - \tau ) D_h ( x_{k - 1} ( r^{-1} \alpha_k ), x_{k - 1} ) . \notag
\end{equation}
If $x_{k - 1} ( r^{-1} \alpha_k ) \neq x_{k - 1}$, we get
\begin{equation}
\alpha_k > \frac{r ( 1 - \tau )}{L_{x_\infty}} , \notag
\end{equation}
a contradiction. 
Therefore, $\liminf \set{ \alpha_k }$ is strictly positive, and the lemma follows. 
\end{proof}

\begin{proof}[Proof of Statements 2 and 3 of Theorem \ref{thm_main_MD}]
If none of the iterates is a solution to \eqref{eq_problem_general}, Theorem \ref{thm_base} and Lemma \ref{lem_lower_bound} imply that the sequence $\set{ f ( x_k ) }$ converges to $f^\star$. 
Otherwise, if $x_k$ is a solution, Lemma \ref{lem_fixed_point_MD} implies that $x_{k'} = x_k$ for every $k' > k$. 
Monotonicity of the sequence $\set{ f ( x_k ) }$ follows from Corollary \ref{cor_monotonicity} in Section \ref{app_auxiliary_MD}. 
\end{proof}

\section{Numerical Results} \label{sec_numerical}

We illustrate applications of Theorem \ref{thm_main_MD} in this section. 

\subsection{Portfolio Selection}

Consider long-term investment in a market of $d$ stocks under the discrete-time setting. 
At the beginning of the $t$-th day, $t \in \mathbb{N}$, the investor distributes his total wealth to the stocks following a vector $x_t$ in the probability simplex $\mathcal{P} \subset \mathbb{R}^d$. 
Denote the price relatives---(possibly negative) returns the investor would receive at the end of the day with one-dollar investment---of the stocks by a vector $a_t \in [ 0, + \infty )^d$. 
% At the end of the $t$-th day, the total wealth of the investor increases by a factor of $\braket{ a_t, x_t }$. 
Then, if the investor has one dollar at the beginning of the first day, the wealth at the end of the $t$-th day is $\Pi_{i = 1}^t \braket{ a_i, x_i }$. 
For every $t \in \mathbb{N}$, the \emph{best constant rebalanced portfolio} $x_t^\star$ up to the $t$-th day is defined as a solution of the optimization problem \cite{Cover1991}
\begin{equation}
x^\star \in \argmin_x \Set{ - \sum_{i = 1}^t \log \braket{ a_i, x } | x \in \mathcal{P} }. \tag{BCRP} \label{eq_bcrp}
\end{equation}
The wealth incurred by the best constant rebalanced portfolio is a benchmark for on-line portfolio selection algorithms \cite{Cover1991,Cover1996,Hazan2015}. 

Denote the objective function in \eqref{eq_bcrp} by $f_{\text{BCRP}}$. 
As $f_{\text{BCRP}}$ is simply a vector analogue of $f_{\text{QST}}$, most existing convergence guarantees in convex optimization does not hold. 
The optimization problem \eqref{eq_bcrp} was addressed by an expectation-maximization (EM)-type method developed by Cover \cite{Cover1984}. 
Given an initial iterate $x_0 \in \mathcal{P} \cap \dom ( f_{\text{BCRP}} )$, Cover's algorithm iterates as
\begin{equation}
x_k = - x_{k - 1} \cdot \nabla f_{\text{BCRP}} ( x_{k - 1} ), \quad \forall k \in \mathbb{N} , \notag
\end{equation}
where the symbol ``$\cdot$'' denotes element-wise multiplication. 
The algorithm possesses a guarantee of convergence but not the convergence rate \cite{Cover1984,Csiszar1984}. 

Now we show that the optimization problem \eqref{eq_bcrp} can be also solved by the exponentiated gradient method with Armijo line search. 

\begin{proposition} \label{prop_bcrp}
The function $f_{\text{BCRP}}$ is locally smooth relative to the (negative) Shannon entropy on the constraint set $\mathcal{P}$. 
\end{proposition}

\begin{proof}
Note that $\dom ( f_{\text{BCRP}} )$ is open, and hence $\dom ( f_{\text{BCRP}} ) \cap \mathcal{X}$ is relatively open in $\mathcal{X}$. 
It is easily checked that $f_{\text{BCRP}}$ is twice continuously differentiable on $\dom ( \fqst )$, and hence on $\dom ( f_{\text{BCRP}} ) \cap \mathcal{X}$.
By Proposition \ref{prop_twice_differentiability}, the function $f_{\text{BCRP}}$ is locally smooth relative to $h( \cdot ) := ( 1 / 2 ) \norm{ \cdot }_{2}^2$. 
By Pinsker's inequality \cite{Csiszar2011}, the Shannon entropy is strongly convex on $\mathcal{P}$ with respect to the $\ell_1$-norm. 
As all norms on a finite-dimensional space are equivalent, the proposition follows. 
\end{proof}

Therefore, the exponentiated gradient method---mirror descent with the Shannon entropy---is guaranteed to converge for solving \eqref{eq_bcrp}. 
The iteration rule has a closed-form: 
\begin{equation}
x ( \alpha ) = c^{-1} x \cdot \exp ( - \alpha \nabla f_{\text{BCRP}} ( x ) ) \, , \quad \forall x \in \mathcal{P} , \alpha \geq 0 \, , \notag
\end{equation}
where we set $\exp ( v ) := ( \eu^{v_1}, \ldots, \eu^{v_d} )$ for any $v = ( v_1, \ldots, v_d ) \in \mathbb{R}^d$. 

\begin{figure}[t]
\centering
\includegraphics[width=.8\textwidth]{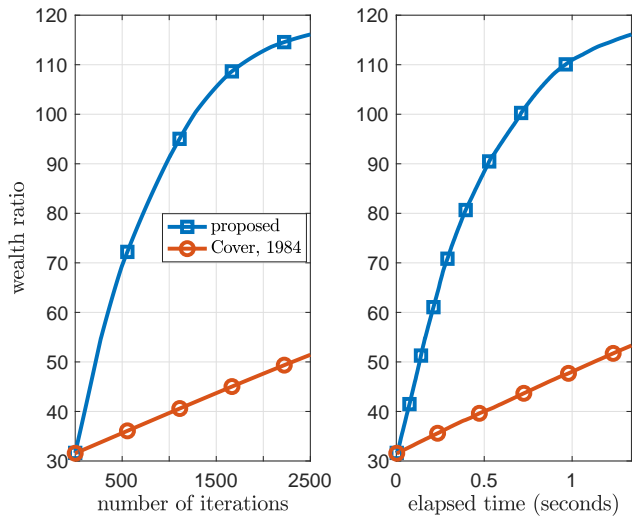} 
\caption[Comparison of convergence rates for portfolio selection]{\label{fig_portfolio} Wealth yielded by different algorithms on the NYSE data.}
\end{figure}

We compare the convergence speeds of Cover's algorithm and the exponentiated gradient method with Armijo line search, for the New York Stock Exchange (NYSE) data during January 1st, 1985--June 30th, 2010 \cite{Li2016b}.
The corresponding dimensions are $n = 6431$ and $d = 23$. 
We set $\overline{\alpha} = 10$, $r = 0.5$, and $\tau = 0.8$ for the Armijo line search procedure. 
The numerical experiment was done in MATLAB R2018a, on a MacBook Pro with an Intel Core i7 2.8GHz processor and 16GB DDR3 memory. 

The numerical result is presented in Figure \ref{fig_portfolio}, where we plot the total wealth yielded by the algorithm iterates, with an initial wealth of one dollar. 
The proposed approach---exponentiated gradient method with Armijo line search---was obviously faster than Cover's algorithm. 
For example, fixing the budget of the computation time to be one second, the proposed approach yields more than twice of the wealth yielded by Cover's algorithm. 
% We note that we did not manage to find the best parameters for Armijo line search; we only tried two sets of parameters. 

\subsection{Quantum State Tomography}

Quantum state tomography (QST) is the task of estimating the state of qubits (quantum bits), given measurement outcomes. 
Numerically, QST corresponds to solving a convex optimization problem specified in Example \ref{exp_qst}. 
Recall that in the introduction, we have shown that the corresponding objective function, $f_{\text{QST}}$, does not satisfy the bounded gradient condition and is not smooth relative to the von Neumann entropy, while mirror descent with the Burg entropy has high per-iteration computational complexity. 

The following proposition is a matrix analogue to Proposition \ref{prop_bcrp}. 
A proof is provided in Section \ref{proof_qst}. 

\begin{proposition} \label{prop_qst}
The function $f_{\text{QST}}$ is locally smooth relative to the von Neumann entropy on the constraint set $\mathcal{D}$. 
\end{proposition}

Therefore, the \emph{(matrix) exponentiated gradient method}---mirror descent with the von Neumann entropy---with Armijo line search is guaranteed to converge, by Theorem \ref{thm_main_MD}. 
The corresponding iteration rule has a closed-form expression \cite{Bubeck2015a,Tsuda2005}: 
\begin{equation}
x ( \alpha ) = c^{-1} \exp ( \log ( x ) - \alpha \nabla f ( x ) ) , \notag
\end{equation}
for every $x \in \tilde{\mathcal{X}}$ and $\alpha \geq 0$, where $c$ is a positive real normalizing the trace of $x( \alpha )$. 
The functions $\exp$ and $\log$ denote matrix exponential and logarithm, respectively. 

We test the empirical performance of the exponentiated gradinet method with Armijo line search, on real experimental data generated following the setting in \cite{Haffner2005}. 
We compare it with the performances of the diluted $R \rho R$ algorithm \cite{Rehacek2007}, SCOPT \cite{Tran-Dinh2015b}, and the modified Frank-Wolfe algorithm studied in \cite{Odor2016}. 
% To the best of our knowledge, the diluted $R \rho R$ algorithm \cite{Rehacek2007}, SCOPT \cite{Tran-Dinh2015b}, and the Frank-Wolfe algorithm studied in \cite{Odor2016} are the only existing algorithms that are guaranteed to converge.
We also consider the $R \rho R$ algorithm \cite{Hradil1997}; it does not always converge \cite{Rehacek2007}, but is typically much faster than the diluted $R \rho R$ algorithm in practice.

\begin{figure}[ht]
\centering
\includegraphics[width=.7\textwidth]{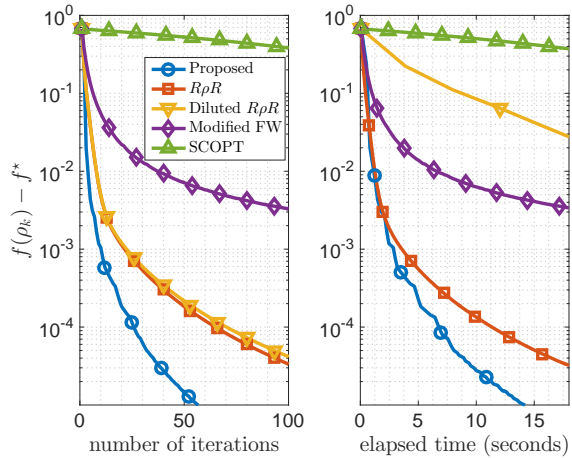} 
\caption[Comparison of convergence rates for QST with 6 qubits]{\label{fig_q6} The 6-qubit case.}
\end{figure}

\begin{figure}[ht]
\centering
\includegraphics[width=.7\textwidth]{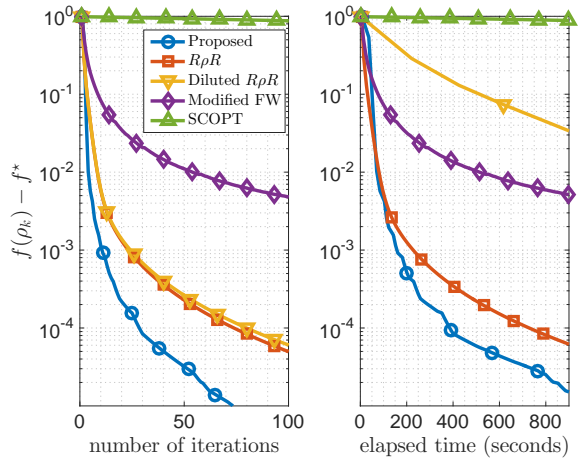} 
\caption[Comparison of convergence rates for QST with 8 qubits]{\label{fig_q8} The 8-qubit case.}
\end{figure}

%\begin{figure}[t]
%    \centering
%    \begin{minipage}{.5\textwidth}
%        \centering
%        \includegraphics[width=\textwidth]{images/figs_mirror_descent/6qubit_D.eps} 
%		\caption[Comparison of convergence rates for QST with 6 qubits]{\label{fig_q6} The 6-qubit case.}
%    \end{minipage}%
%    \begin{minipage}{0.5\textwidth}
%        \centering
%        \includegraphics[width=\textwidth]{images/figs_mirror_descent/8qubit_D.eps} 
%		\caption[Comparison of convergence rates for QST with 8 qubits]{\label{fig_q8} The 8-qubit case.}
%    \end{minipage}
%\end{figure}

We compare the convergence speeds for the $6$-qubit ($d = 2^6$) and $8$-qubit ($d = 2^8$) cases, in Fig. \ref{fig_q6} and \ref{fig_q8}, respectively.
The corresponding ``sample sizes'' (number of summands in $\fqst$) are $n = 60640$ and $n = 460938$, respectively.
The numerical experiments were done in MATLAB R2015b, on a MacBook Pro with an Intel Core i7 2.8GHz processor and 16GB DDR3 memory.
We set $\alpha = 10$, and $\gamma = \tau = 0.5$ in Algorithm \ref{alg_MD} for both cases.
In both figures, $f^\star$ denotes the minimum value of $\fqst$ found by the five algorithms in 120 iterations.

One can observe that the exponentiated gradient method with Armijo line search is the fastest, in terms of the actual elapsed time. 
The slowness of the other algorithms is explainable. 

\begin{enumerate}
\item The diluted $R \rho R$ algorithm, using the notation of this paper, iterates as 
\begin{equation}
x_{k + 1} = c_k^{-1} \left[ I + \beta_k f' ( x_k ) \right]^{\mathrm{H}} \rho_k \left[ I + \beta_k f' ( x_k ) \right], \notag
\end{equation}
where $c_k$ normalizes the trace of $x_{k + 1}$. 
To guarantee convergence, the step sizes $\beta_k$ are computed by exact line search.
The exact line search procedure renders the algorithm slow.

\item SCOPT is a projected gradient method for minimizing self-concordant functions \cite{Nesterov2004,Nesterov1994}. 
Notice that projection onto $\mathcal{D}$ typically results in a low-rank output; hence, it is possible that $\tr ( M_i x_k ) = 0$ for some low-rank  $M_i$ and iterate $x_k$, but then $x_k$ is not a feasible solution because $\log ( 0 )$ is not defined\footnote{In a standard setup of quantum state tomography, the matrices $M_i$ are single-rank \cite{Haffner2005}.}. 
This is called the stalling problem in \cite{Knee2018}. 
Luckily, self-concordance of $f_{\text{QST}}$ ensures that if an iterate $x_k$ lies in $\dom f_{\text{QST}}$, and the next iterate $x_{k + 1}$ lies in a small enough \emph{Dikin ellipsoid} centered at $x_k$, then $x_{k + 1}$ also lies in $\dom f_{\text{QST}}$. 
It is easily checked that $\fqst$ is a self-concordant function of parameter $2 \sqrt{ n }$.  
Following the theory in \cite{Nesterov2004,Nesterov1994}, the radius of the Dikin ellipsoid shrinks at the rate $O( n^{-1/2} )$, so SCOPT becomes slow when $n$ is large.

\item The Frank-Wolfe algorithm suffers for a sub-linear convergence rate when the solution is near an extreme point of the constraint set (see, e.g., \cite{Lacoste-Julien2015} for an illustration in the vector case). 
Notice that the set of extreme points of $\mathcal{D}$ is the set of single-rank positive semi-definite matrices of unit trace. 
In the experimental data we have, the density matrix to be estimated is indeed close to a single-rank matrix (which is called a \emph{pure state} in quantum mechanics). 
Therefore, the ML estimate---the minimizer of $f_{\text{QST}}$ on $\mathcal{D}$---is expected to be also close to a single-rank matrix. 
\end{enumerate}

Notice that the empirical convergence rate of the exponentiated gradient method with Armijo line search is linear. 
% This is a typical result of minimizing a strongly convex function by first-order optimization algorithms; interestingly, the function $f_{\text{QST}}$ is not strongly convex.

\section*{Acknowledgements}
We thank David Gross for valuable discussions, and Ya-Ping Hsieh for checking previous versions of this paper. 
YHL and VC were supported by SNF 200021-146750 and ERC project time-data 725594. 
CAR was supported by the Freie Universit\"{a}t Berlin within the Excellence Initiative of the German Research Foundation, DFG (SPP 1798 CoSIP), and the Templeton Foundation.

\appendix

\section{Proof of Proposition \ref{prop_not_applicable}} \label{proof_not_applicable}

Consider the two-dimensional case, where $x = ( x_{i, j} )_{1 \leq i, j \leq 2} \in \mathbb{C}^{2 \times 2}$. 
Define $e_1 := ( 1, 0 )$ and $e_2 := ( 0, 1 )$. 
Suppose that there are only two summands, with $M_1 = e_1 \otimes e_1$ and $M_2 = e_2 \otimes e_2$. 
Then we have $f ( x ) = - \log ( x_{1,1} ) - \log ( x_{2,2} )$. 
It suffices to disprove all properties on the set of diagonal density matrices.
Hence, we will focus on the function $g ( x, y ) := - \log x - \log y$, defined for any $( x, y )$ in the probability simplex $\mathcal{P} \subset \mathbb{R}^2$.

As either $x$ or $y$ can be arbitrarily close to zero, it is easily checked that the gradient of $g$ is unbounded. 
Now we check the relative smoothness condition. 
As we only consider diagonal matrices, it suffices to check with respect to the (negative) Shannon entropy: 
\begin{equation}
h(x, y) := - x \log x - y \log y + x + y \, , \quad \forall ( x, y ) \in \mathcal{P} \, , \notag
\end{equation}
for which the convention $0 \log 0 := 0$ is adopted. 

\begin{lemma}[\cite{Lu2018}]
The function $g$ is $L$-smooth relative to the Shannon entropy for some $L > 0$, if and only if $- L h - g$ is convex.
\end{lemma}

Therefore, we check the positive semi-definiteness of the Hessian of $- L h - g$. 
A necessary condition for the Hessian to be positive semi-definite is that 
\begin{equation}
- L \frac{\partial^2 h}{\partial x^2} ( x, y ) - \frac{\partial^2 g}{\partial x^2} ( x, y ) = \frac{L}{x} - \frac{1}{x^2} \geq 0 , \notag
\end{equation}
for all $x \in ( 0, 1 )$, but the inequality cannot hold for $x < ( 1 / L )$, for any fixed $L > 0$.  

\section{Proof of Lemma \ref{lem_equiv}} \label{app_lem_equiv}

(Statement 2 $\Rightarrow$ Statement 1) 
Let $x \in \mathcal{X} \cap \dom f$, and $z_1, z_2 \in \mathcal{B}_{\varepsilon_x} ( x ) \cap \tilde{\mathcal{X}}$. 
Define, for every $\tau \in [ 0, 1 ]$, $z_\tau := z_1 + \tau ( z_2 - z_1 )$. 
We write
\begin{align}
f ( z_2 ) - \left[ f ( z_1 ) + \braket{ \nabla f ( z_1 ), z_2 - z_1 } \right] 
& = \int_0^1 \braket{ \nabla f ( z_\tau ) - \nabla f ( z_1 ), z_2 - z_1 } \, \du \tau \notag \\
& \leq \int_0^1 \norm{ \nabla f ( z_\tau ) - \nabla f ( z_1 ) }_2 \norm{ z_2 - z_1 }_2 \, \du \tau \notag \\
& \leq \int_0^1 L_x \tau \norm{ z_2 - z_1 }_2^2 \, \du \tau \notag \\
& = \frac{L_x}{2} \norm{ z_2 - z_1 }_2^2 , \notag
\end{align}
where we have applied the Cauchy-Schwarz inequality for the first inequality, and the local smoothness condition for the second inequality. 
Note that $\mathcal{B}_{\varepsilon_x} \cap \tilde{\mathcal{X}}$ is the intersection of convex sets, and hence is convex; therefore, $z_\tau \in \mathcal{B}_{\varepsilon_x} \cap \tilde{\mathcal{X}}$ for every $\tau \in [ 0, 1 ]$. 

(Statement 1 $\Rightarrow$ Statement 2) 
Let $x \in \mathcal{X} \cap \dom f$, and $z_1, z_2 \in \mathcal{B}_{\varepsilon_x} ( x ) \cap \tilde{\mathcal{X}}$. 
Define $\varphi ( z ) := f ( z ) - \braket{ \nabla f ( z_1 ), z }$. 
Then $\nabla \varphi$ is locally Lipschitz on $\tilde{\mathcal{X}}$; moreover, since $\nabla \varphi ( z_1 ) = 0$, the point $z_1$ is a global minimizer of $\varphi$. 
Therefore, we obtain
\begin{equation}
\varphi ( z_1 ) \leq \varphi ( z_2 - \frac{1}{L_x} \nabla \varphi ( z_2 ) ) \leq \varphi ( z_2 ) - \frac{1}{2 L_x} \norm{ \nabla \varphi ( z_2 ) }^2 ; \notag
\end{equation}
that is, 
\begin{equation}
f ( z_2 ) \geq f ( z_1 ) + \braket{ \nabla f ( z_1 ), z_2 - z_1 } + \frac{1}{2 L_x} \norm{ \nabla f ( z_2 ) - \nabla f ( z_1 ) }_2^2 . \notag
\end{equation}
Similarly, we get
\begin{equation}
f ( z_1 ) \geq f ( z_2 ) + \braket{ \nabla f ( z_2 ), z_1 - z_2 } + \frac{1}{2 L_x} \norm{ \nabla f ( z_1 ) - \nabla f ( z_2 ) }_2^2 . \notag
\end{equation}
Summing up the two inequalities; we obtain
\begin{equation}
\braket{ \nabla f ( z_2 ) - \nabla f ( z_1 ), z_2 - z_1 } \geq \frac{1}{L_x} \norm{ \nabla f ( z_2 ) - \nabla f ( z_1 ) }_2^2 . \notag
\end{equation}
This implies, by the Cauchy-Schwarz inequality, 
\begin{equation}
\norm{ \nabla f ( z_2 ) - \nabla f ( z_1 ) }_2 \leq L_x \norm{ z_2 - z_1 }_2 . \notag
\end{equation}

\section{Auxiliary Technical Lemmas for Proving Theorem \ref{thm_main_MD}} \label{app_auxiliary_MD}

\begin{lemma} \label{lem_fixed_point_MD}
If $x ( \alpha ) = x$ for some $x \in \tilde{\mathcal{X}}$, then $x$ is a solution to \eqref{eq_problem_general}. 
If a point $x \in \tilde{\mathcal{X}}$ is a solution to \eqref{eq_problem_general}, then $x ( \alpha ) = x$ for all $\alpha \in [ 0, + \infty)$. 
\end{lemma}

\begin{proof}
That $x$ is a solution to \eqref{eq_problem_general} is equivalent to the optimality condition
\begin{equation}
\braket{ \nabla f ( x ), z - x } \geq 0 , \quad \forall z \in \mathcal{X} . \notag
\end{equation}
We can equivalently write
\begin{equation}
\braket{ \alpha \nabla f ( x ) + \nabla h ( x ) - \nabla h ( x ), z - x } \geq 0 , \quad \forall z \in \mathcal{X} , \notag
\end{equation}
which is the optimality condition of 
\begin{equation}
x ( \alpha ) = \argmin_{z} \set{ \alpha \braket{ \nabla f ( x ), z - x } + D_h ( z, x ) | z \in \mathcal{X} } . \notag
\end{equation}
\end{proof}

\begin{lemma} \label{lem_monotonicity}
For every $x \in \tilde{\mathcal{X}}$ and $\alpha > 0$, it holds that 
\begin{equation}
\braket{ \nabla f ( x( \alpha ) ), x( \alpha ) - x } \leq - \alpha^{-1} D ( x( \alpha ), x ) \leq 0 . \notag
\end{equation}
\end{lemma}

\begin{proof}
By definition, we have
\begin{equation}
\alpha \braket{ \nabla f ( x( \alpha ) ), x( \alpha ) - x } + D ( x( \alpha ), x ) \leq \alpha \braket{ \nabla f ( x, x - x } + D ( x, x ) = 0 . \notag
\end{equation}
\end{proof}

\begin{corollary} \label{cor_monotonicity}
The sequence $\set{ x_k }$ is non-increasing. 
\end{corollary}

\begin{proof}
The Armijo rule and Lemma \ref{lem_monotonicity} guarantee that 
\begin{equation}
f ( x_k ) \leq f ( x_{k - 1} ) + \tau \braket{ \nabla f ( x_{k - 1} ), x_k - x_{k - 1} } \leq f ( x_{k - 1} ) . \notag
\end{equation}
\end{proof}

\section{Proof of Theorem \ref{thm_base}} \label{app_thm_base}

For every $u \in \mathcal{X} \cap \dom f$, we write
\begin{align}
f ( x_{k - 1} ) - f ( u ) & \leq - \braket{ \nabla f ( x_{k - 1} ), u - x_{k - 1} } \notag \\
& = - \braket{ \nabla f ( x_{k - 1} ), u - x_k } - \braket{ \nabla f ( x_{k - 1} ), x_k - x_{k - 1}} . \notag
\end{align}
The optimality condition for $x_k$ implies
\begin{equation}
\braket{ \alpha_k \nabla f ( x_{k - 1} ) + \nabla h ( x_k ) - \nabla h ( x_{k - 1} ), u - x_k } \geq 0 . \notag
\end{equation}
Applying the \emph{three-point identity} \cite{Chen1993}, we obtain
\begin{align}
\braket{ \nabla f ( x_{k - 1} ), u - x_k } & \geq - \alpha_k^{-1} \braket{ \nabla h ( x_k ) - \nabla h ( x_{k - 1} ), u - x_k } \notag \\
& = - \alpha_k^{-1} \left[ D_h ( u, x_{k - 1} ) - D_h ( u, x_k ) - D_h ( x_k, x_{k - 1} ) \right] \notag \\
& \geq - \alpha_k^{-1} \left[ D_h ( u, x_{k - 1} ) - D_h ( u, x_k ) ) \right] . \notag
\end{align}
Then we can write 
\begin{align}
\alpha_k \left[ f ( x_{k - 1} ) -  f ( u ) \right] & \leq \left[ D_h ( u, x_{k - 1} ) - D_h ( u, x_k ) ) \right] - \alpha_k \braket{ \nabla f ( x_{k - 1} ), x_k - x_{k - 1} } . \notag
\end{align}
Summing up the inequality for all $1 \leq k \leq n$, we get 
\begin{align}
- S_n f ( u ) + \sum_{k = 1}^n \alpha_k f ( x_{k - 1} ) \leq D ( u, x_0 ) - \sum_{k = 1}^n \alpha_k \braket{ \nabla f ( x_{k - 1} ), x_k - x_{k - 1} } , \notag
\end{align}
where $S_n := \sum_{k = 1}^{n} \alpha_k$. 
Corollary \ref{cor_monotonicity} says that the sequence $( f ( x_k ) )_{k \in \mathbb{N}}$ is non-increasing; then we have
\begin{equation}
\sum_{k = 1}^n \alpha_k f ( x_{k - 1} ) \geq \sum_{k = 1}^n \alpha_k f ( x_n ) = S_n f ( x_n ) . \notag
\end{equation}
Therefore, we obtain
\begin{align}
f ( x_n ) - f ( u ) \leq S_n^{-1} \left[ D ( u, x_0 ) - \sum_{k = 1}^n \alpha_k \braket{ \nabla f ( x_{k - 1} ), x_k - x_{k - 1} } \right] . \notag
\end{align}

Note that by the Armijo rule, we have
\begin{align}
f ( x_0 ) - f^\star & \geq \lim_{k \to \infty} f ( x_0 ) - f ( x_k ) \notag \\
& = \sum_{j = 1}^\infty \left[ f ( x_{j - 1} ) - f ( x_{j} ) \right] \notag \\
& \geq - \tau \sum_{j = 1}^\infty \braket{ \nabla f ( x_{j - 1} ), x_j - x_{j - 1}  } . \notag
\end{align}
Therefore, $\braket{ \nabla f ( x_{k - 1} ), x_k - x_{k - 1} }$, which are non-negative by Lemma \ref{lem_lower_bound}, must converge to zero. 
Theorem \ref{thm_base} then follows from the following lemma. 

\begin{lemma}[\cite{Polyak1987}]
Let $\set{ a_k }$ be a sequence of real numbers, and $\set{ b_k }$ be a sequence of positive real numbers. 
Define $c_n := \sigma_n^{-1} \sum_{k = 1}^n b_k a_k$ for every $n \in \mathbb{N}$, where $\sigma_n := \sum_{k = 1}^n b_k$. 
If $a_k \to 0$ and $\sigma_n \to + \infty$, then $c_n \to 0$. 
\end{lemma}

\section{Proof of Proposition \ref{prop_qst}} \label{proof_qst}

Note that $\dom ( f_{\text{QST}} )$ is open, and hence $\dom ( f_{\text{QST}} ) \cap \mathcal{X}$ is relatively open in $\mathcal{X}$. 
It is easily checked that $f_{\text{QST}}$ is twice continuously differentiable on $\dom ( \fqst )$, and hence on $\dom ( \fqst ) \cap \mathcal{X}$.
By Proposition \ref{prop_twice_differentiability}, the function $\fqst$ is locally smooth relative to $h( \cdot ) := ( 1 / 2 ) \norm{ \cdot }_{\text{F}}^2$, where $\norm{ \cdot }_{\text{F}}$ denotes the Frobenius norm. 
By the quantum version of Pinsker's inequality \cite{Hiai1981}, the von Neumann entropy is strongly convex on $\mathcal{D}$ with respect to the trace norm. 
As all norms on a finite-dimensional space are equivalent, the proposition follows. 

\bibliographystyle{acm}
\bibliography{list}

\end{document}